\documentclass[12pt]{article}
\usepackage[centertags]{amsmath}
\usepackage{amsfonts}
\usepackage{amssymb}
\usepackage{amsthm}
\input{amssym.def}

\numberwithin{equation}{section}

\newtheorem{Definition}{Definition}[section]
\newtheorem{theorem}[Definition]{Theorem}
\newtheorem{lemma}[Definition]{Lemma}

\begin{document}
\title{\Large \bf An elementary proof of a conjecture on graph-automorphism }
\author{ Sajal Kumar Mukherjee and A. K. Bhuniya  }
\date{}
\maketitle

\begin{center}
Department of Mathematics, Visva-Bharati, Santiniketan-731235, India. \\
shyamal.sajalmukherjee@gmail.com, anjankbhuniya@gmail.com
\end{center}

\begin{abstract}
In this article, we give an elementary combinatorial proof of a conjecture about the determination of automorphism group of the power graph of finite cyclic groups, proposed by Doostabadi, Erfanian and Jafarzadeh in 2013.
\end{abstract}
\section{Introduction}
Let $G$ be a finite group. The concept of directed power graph $\overrightarrow{\mathcal{P}(G)}$ was introduced by Kelarev and Quinn \cite{K}. $\overrightarrow{\mathcal{P}(G)}$ is a digraph with vertex set $G$ and for $x,y \in G$, there is an arc from $x$ to $y$ if and only if $x\neq y$ and $y = x^{m}$ for some positive integer $m$. Following this Chakrabarty, Ghosh and Sen \cite{pet} defined undirected power graph $\mathcal{P}(G)$ of a group $G$ as an undirected graph with vertex set $G$ and two distinct vertices are adjacent if and only if one of them is a positive power of the other.

In 2013, Doostabadi, Erfanian and Jafarzadeh \cite{D} conjectured that for any natural number $n$,

\noindent
Aut $(\mathcal{P}(\mathbb{Z}_n))=(\bigoplus_{d|n, d\neq 1, n}S_{\phi(d)})\bigoplus S_{\phi(n)+1}$, where $\phi$ is the Euler's phi function. Although, if $n$ is a prime power, then $\mathcal{P}(\mathbb{Z}_n)$ is complete \cite{pet}, hence Aut $(\mathcal{P}(\mathbb{Z}_n))= S_{n}$. Hence, the conjecture does not hold if $n = p^{m}$. In June 2014, Min Feng, Xuanlong Ma, Kaishun Wang \cite{s} proved that the conjecture holds for the remaining cases, that is for $n\neq p^{m}$. In fact they proved a more general result, but their proof uses some what complicated group theoritic arguments. Our aim of this paper is to provide a much more elementary combinatorial proof of the conjecture for $n\neq p^{m}$ without using any nontrivial group theoritic result.

\section{ Main Theorem}
In this section, first we prove several lemmas and as a consequence, we shall prove the following result, which is the main theorem theorem of this article.
\begin{theorem}
For $n\neq p^{m}$ ($p$ prime),
Aut $(\mathcal{P}(\mathbb{Z}_n))=(\bigoplus_{d|n, d\neq 1, n}S_{\phi(d)})\bigoplus S_{\phi(n)+1}$
\end{theorem}

First we prove a technical lemma, which is also the heart of our argument. Before stating it, we have to fix some notations.

Let $S$ be a finite set of positive real numbers. For each subset $B\subseteq S$, let $\prod(B)$ denote the product of all the elements of $B$. Now  let us state and prove the lemma.
\begin{lemma}
Let $n\geq 2$ and $m_1, m_2, \cdots m_n$ be  $n$ positive integers with $m_1>m_2> \cdots >m_n$. Let $m$ be any positive integer, and set $A=\{m_1, m_2, \cdots, m_n; m \}$ and $B=\{m_2, m_3, \cdots, m_n \}$. Then to every non empty subset $S_B$ of $B$, we can associate a proper subset $S_A$ of $A$, for which $m_1, m \in S_A$ and  $\prod(S_B)<\prod(S_A\setminus \{m\})$. The association can be made one to one.
\end{lemma}
\begin{proof}
We prove this by induction on $n$. Let $n=2$.
\noindent
Then $A=\{m_1, m_2, m\}$,  $B=\{m_2\}$ and $m_1>m_2$. The only nonempty subset of $B$
is $B$ itself. To $B$, we associate $\{m_1, m\}$ and the result holds.

Now let the statement be true for $n=k$. We will prove for $n=k+1$. Let $m_1, m_2, \cdots m_{k+1}$ be any $k+1$ positive integers with $m_1>m_2> \cdots > m_{k+1}$ and $m$ be any positive integer. Then $A=\{m_1,m_2, \cdots, m_{k+1}; m\}$ and $B=\{m_2, m_3, \cdots m_{k+1}\}$. Let $B_1$ be the collection of non empty subsets of $B$, which do not contain $m_{k+1}$ and $B_2$ be the collection of subsets of $B$ containing $m_{k+1}$. Now consider $\acute{A}=\{m_1, m_2, \cdots m_k; m_{k+1}\}$ and $\acute{B}=\{m_2,m_3 \cdots m_k\}$. Then by the induction hypothesis, to each nonempty subsets $S_{\acute{B}}$ of $\acute{B}$ we can associate a proper subset $S_{\acute{A}}$ of $\acute{A}$, for which $m_1, m_{k+1}\in S_{\acute{A}}$ and $\prod(S_{\acute{B}})<\prod(S_{\acute{A}}\setminus\{m_{k+1}\})$. Now let $S_{B_1}$ be an arbitrary element of $B_1$. But $S_{B_1}$ is also a non empty subset of $\acute{B}$. So we have a proper subset $S_{\acute{A_1}}$ of $\acute{A}$ for which $m_1, m_{k+1}\in S_{\acute{A_1}}$ and $\prod(S_{B_1})<\prod(S_{\acute{A_1}}\setminus\{m_{k+1}\})$ [by induction hypothesis]. Set $S_{A_1}(\subsetneq A)$ to be $(S_{\acute{A_1}}\setminus\{m_{k+1}\})\bigcup\{m\}$. Clearly $\prod(S_{B_1})<\prod(S_{A_1}\setminus\{m\})$. Now let $S_{B_2}$ be an arbitrary element of $B_2\setminus\{\{m_{k+1}\}\}$. Then $S_{B_2}\setminus\{m_{k+1}\}\in B_1$. So for the sake of simplicity assume that $S_{B_2}\setminus\{m_{k+1}\}=S_{B_1}$. But for $S_{B_1}$, we have $S_{A_1}$, so that $\prod(S_{B_1})<\prod(S_{A_1}\setminus\{m\})$. Let us take $S_{A_2}$ to be $S_{A_1}\bigcup \{m_{k+1}\}$. It is easy to see that $S_{A_2}$ is infact a proper subset of $A$ and $\prod(S_{B_2})<\prod(S_{A_2}\setminus\{m\})$. Now the number of non empty subset of $B$ is equal to the number of proper subset of $A$ containing both $m_1, m$ is equal to $2^k-1$. Hence  for the set $\{m_{k+1}\}$, there still remains exactly one proper subset $S_A$ of $A$, containing both $m_1, m_2$. Clearly $\prod(\{m_{k+1}\})=m_{k+1}<m_1\leq \prod(S_A\setminus\{m\})$. This completes the induction as well as the proof.
\end{proof}

Before plunging into a chain of lemmas, we once again fix some notations. Let $X_d$ denote the set of all generators of the unique cyclic subgroup of $\mathbb{Z}_n$ of order $d$. Then $|X_d|=\phi(d)$ We will denote a general element of $X_d$ by $x_d$.

\begin{lemma}
If $n=p_1p_2 \cdots p_k$ where $p_1>p_2 \cdots >p_k$ are distinct primes then there does not exist any graph automorphism $\sigma \in$ Aut $(\mathcal{P}(\mathbb{Z}_n))$ such that $\sigma(x_{p_i})=x_{p_{i_1}p_{i_2}\cdots p_{i_l}}$ for any $i, l$ and $i_1, i_2, \cdots i_l$ satisfying $i\geq i_1>i_2>\cdots >i_l$.
\end{lemma}
\begin{proof}
To prove the lemma, it suffices to prove that degree $(x_{p_i})>$ degree $(x_{p_{i_1}p_{i_2}\cdots p_{i_l}})$, satisfying the above stated condition. But, since degree$(x_{p_i})\geq$ degree$(x_{p_{i_1}})$ [as $i \geq {i_1}$], it suffices to prove that degree $(x_{p_{i_1}})>$ degree $(x_{p_{i_1}p_{i_2}\cdots p_{i_l}})$. We will only prove that degree $(x_{p_1})>$ degree $(x_{p_1p_2 \cdots p_{k-1}})$ for the sake of simplicity, because the other cases will follow exactly in the similar fashion. Now, in Lemma $3.1$, take

$A = \{ p_{1}-1,   p_{2}-1,   \cdots  p_{k-1}-1,   p_{k}-1 \}$  and $B = \{ p_{2}-1,   p_{3}-1,   \cdots  p_{k-1}-1 \}$. Then from Lemma $3.1$ we see that the total number of vertices adjacent to $x_{p_1}$ but not adjacent to $x_{p_1p_2\cdots p_{k-1}}$ is strictly greater than the number of vertices adjacent to $x_{p_1p_2\cdots p_{k-1}}$ but not adjacent to $x_{p_1}$. And this proves the claim.
\end{proof}

\begin{lemma}
Let $n=p_1^{m_1}p_2^{m_2} \cdots p_k^{m_k}$ and $p_1^{m_1}>p_2^{m_2}> \cdots p_k^{m_k}$. Then we have :

(i): There does not exist any automorphism $\sigma$ such that $\sigma(x_{p_i})=x_{p_{i_1}^{\alpha_1}p_{i_2}^{\alpha_2 }\cdots p_{i_r}^{\alpha_r}}$ for any $i, r$ and a proper subset $\{i_1, i_2, \cdots i_r\}$ of $\{1, 2, \cdots k\}$ satisfying $i\geq i_1>i_2 \cdots >i_r$.

(ii): There does not exist any automorphism $\eta$ such that $\eta(x_{p_1})=x_{{p_1}^{\alpha_1}{p_2}^{\alpha_2}\cdots {p_k}^{\alpha}_k}$ for any $\alpha_1, \alpha_2, \cdots \alpha_k$ with $1\leq \alpha_i\leq m_i$ for all $i=1, 2, \cdots k$.
\end{lemma}

\begin{proof}
(i): Proof of this part is exactly similar to that of Lemma $2.3$.

 (ii): We divide this case into two cases.

Case(1): $k\geq 3$,  i.e $n=p_1^{x_1}p_2^{x_2}\cdots p_k^{x_k}$ where $k\geq 3$ and $p_1^{x_1}>p_2^{x_2}> \cdots > p_k^{x_k}$. We will eliminate the following two difficult cases. Rests will follow quite similarly.

Subcase(1.1): If possible let $\eta(x_{p_1})=x_{p_1^{x_1}p_2^{r}p_3^{x_3}\cdots p_k^{x_k}}$ where $1\leq r<x_2$. Then the number of  vertices adjacent to $x_{p_1^{x_1}p_2^{r}p_3^{x_3}\cdots p_m^{x_m}}$ but not adjacent to $x_{p_1}$ is exactly equal to $(\phi(p_2)+\phi(p_2^2)+\cdots +\phi(p_2^r)+1)(\phi(p_3)+\phi(p_3^2)\cdots + \phi(p_3^{x_3})+1)\cdots (\phi(p_k)+\phi(p_k^2)+\cdots+\phi(p_k^{x_k})+1)-1$ which is equal to $p_2^{r}p_3^{x_3}\cdots p_k^{x_k}-1$. Now there are at least $(\phi(p_1)+\phi(p_1^2)+\cdots +\phi(p_1^{x_1}))(\phi(p_2^{r+1})+\phi(p_2^{r+2})+\cdots \phi(p_2^{x_2}))(\phi(p_3)+\phi(p_3^2)+\cdots +\phi(p_3^{x_3})+1)\cdots (\phi(p_{k-1})+\phi(p_{k-1}^2)+\cdots +\phi(p_{k-1}^{x_{k-1}})+1)=(p_1^{x_1}-1)(p_2^{x_2}-p_2^{r})p_3^{x_3}\cdots p_{k-1}^{x_{k-1}}$ number of vertices adjacent to $x_{p_1}$ but not adjacent to $x_{p_1^{x_1}p_2^rp_3^{x_3}\cdots p_k^{x_k}}$. But we see that $(p_1^{x_1}-1)(p_2^{x_2}-p_2^{r})\cdots (p_{k-1}^{x_{k-1}})>p_2^{r}p_3^{x_3}\cdots p_k^{x_k}$ which shows that degree of $x_{p_1}$ is strictly greater than the degree of $\sigma(x_{p_1})$ which is a contradiction.

Subase(1.2): If possible let $\sigma(x_{p_1})=x_{p_1^{m_1}p_2^{x_2}\cdots p_k^{x_k}}$ where $1\leq m_1<x_1$ but degree $(x_{p_1})>$ degree $(x_{p_2})>$ degree $(x_{p_1^{m_1}p_2^{x_2}\cdots p_k^{x_k}})$. [Second inequality follows exactly same way as the previous case ] Hence a contradiction.

Case(2): $k=2$ i.e $n$ is of the form $n=p^a q^b$ where $p^a>q^b$. Now we have the following situations.

Subcase(2.1): $a=1$, hence we may assume that $b\geq 2$(because the case $b=1$ has already been dealt with). Now if possible let there exists an automorphism $\sigma$ such that $\sigma(x_p)=x_{pq^t}, 1<t<b$. Now since $x_q$ is not adjacent to $x_p$, $\sigma(x_q)$ should be some $x_{q^s}$ where $s>t$. But this is impossible, because degree $(x_q)>$ degree $(x_{q^s})$.

Subcase(2.2): $a>1$. Now if possible let $\sigma(x_p)=x_{p^mq^n}$ where $m\leq a$ and $n\leq b$. Let us assume that $n=b$ and also assume that $b>1$(the case $n<b$ can be handled similarly). Then the number of vertices adjacent to $x_{p^mq^b}$ but not adjacent to $x_p$ is $\phi(q)+\phi(q^2)+\cdots +\phi(q^b)=q^b-1$. And the number of vertices adjacent to $x_p$ but not  adjacent to $x_{p^mq^b}$ is $(\phi(p^{m+1})+\phi(p^{m+2})+\cdots + \phi(p^a))(\phi(q)+\phi(q^2)+\cdots +\phi(q^{b-1})+1)=(p^a-p^m)(q^{b-1})$ which can not be equal to $q^b-1$ if $b>1$.

So assume that $b=1$. Now if possible let $\sigma(x_p)=x_{p^mq}$. Now by the same logic as subcase2.1, $\sigma(x_q)$ is of the form $x_{p^s}$ where $s>m$. Now the number of vertices that are adjacent to $x_{p^s}$ but not adjacent to $x_q$ is $p^a-1$ but the number of vertices adjacent to $x_q$ but not adjacent to $x_{p^s}$ is equal to $(q-1)p^{s-1}$ which is not equal to $p^a-1$. Hence a contradiction and the proof is complete.

\end{proof}

Now we state our main lemma, which immediately implies the theorem.

\begin{lemma}
If $\sigma \in$ Aut $(\mathcal{P}(\mathbb{Z}_n))$, then $\sigma(X_d)=X_d$ for all $d(\neq 1, n)$ dividing $n$.
\end{lemma}
\begin{proof}
We will illustrate the proof for $n=p_1p_2p_3 \cdots p_k$. The general case is similar. Suppose if possible $\sigma (x_{d_1})=x_{d_2}$ for some automorphism $\sigma$ and $d_1, d_2$ two distinct non trivial divisors of $n$. Then there exists some prime $p\in\{p_1, p_2, \cdots p_k\}$ such that $p$ divides $d_2$ but does not divide $d_1$. So by the definition of automorphism there exists an automorphism which sends $p$ to a vertex say $v$ adjacent to $d_1$. Note that $v\neq p$. If $p$ is greater than or equal to each prime divisors of $v$ then we get a contradiction from Lemma $2.3$. So there exists at least a prime divisor of $v$ say $q$ which is strictly greater than $p$. Now there exists an automorphism, which sends $q$ to a vertex $v_1$ adjacent to $p$. Again note that $v_1$ is not equal to $q$. And we proceed similarly as before to eventually reach a contradiction using lemma $2.3$.

\end{proof}

\noindent
\textbf{Acknowledgement:} The first author is partially supported by CSIR-JRF grant.

\bibliographystyle{amsplain}

\end{document}